\documentclass[11pt,fleqn]{amsart}

\usepackage{amsmath,amssymb,amsthm}
\usepackage{psfrag}
\usepackage{hyperref}
\usepackage[dvipsnames,usenames]{color}

\usepackage[english]{babel}
\usepackage{amsmath,amsfonts, amsthm,xcolor,amssymb}
\usepackage{paralist}
\numberwithin{equation}{section}
\usepackage{mathtools, mathabx}
\usepackage{xcolor}

\newtheorem{thm}{Theorem}[section]
\newtheorem{lem}[thm]{Lemma}
\newtheorem{cor}[thm]{Corollary}
\newtheorem{prop}[thm]{Proposition}

\newtheorem{defn}[thm]{Definition}
\theoremstyle{definition}
\newtheorem{rem}[thm]{Remark}
\theoremstyle{remark}

\newcommand{\norm}[1]{\left\Vert#1\right\Vert}

\newcommand{\abs}[1]{\left\vert#1\right\vert}

\newcommand{\R}{\mathbb{R}}
\newcommand{\N}{\mathbb{N}}

\newcommand{\res}{\!\!\mathop{\hbox{
                                \vrule height 7pt width .5pt depth 0pt
                                \vrule height .5pt width 6pt depth 0pt}}
                                \nolimits}

\DeclareMathOperator{\dive}{div}

\DeclareMathOperator{\sign}{sign}

\newcommand\restr[2]{{
  \left.\kern-\nulldelimiterspace 
  #1 
  \vphantom{ |} 
  \right|_{#2} 
  }}
{\left\{\begin{array}{@{}l@{}}}{\end{array}\right.}
\makeatletter 
\def\@makefnmark{} 
\makeatother 

\begin{document}

\title[First Dirichlet eigenvalue of the weighted 1-Laplacian]
{{First Dirichlet eigenvalue of the weighted 1-Laplacian operator}}

\author[R. Barbato, J.C. Sabina de Lis, S. Segura de Le\'on]{Rosa Barbato, Jos\'e C. Sabina de Lis and Sergio Segura de Le\'on}

\date{\today}

\address{Rosa Barbato}
\email{{\tt rossellabarbato95@gmail.com}\hfill\break\indent
\it Orcid link: \href{0000-0001-7976-3307}{0000-0001-7976-3307}}

\address{Jos\'{e} C. Sabina de Lis
\hfill \break \indent Departamento de An\'{a}lisis Matem\'{a}tico {\rm and} IUEA,
\hfill \break \indent Universidad de La
Laguna, \hfill \break \indent P. O. Box 456, 38200 -- La Laguna, SPAIN.}
\email{{\tt josabina@ull.edu.es}\hfill\break\indent
\it Orcid link: \href{0000-0002-2378-7614}{0000-0002-2378-7614}}

\address{Sergio Segura de Le\'on
\hfill \break\indent Departament d'An\`alisi Matem\`atica,
Universitat de Val\`encia, \hfill\break\indent Dr. Moliner 50,
46100 Burjassot, Val\`encia, SPAIN.} 
\email{{sergio.segura@uv.es}\hfill\break\indent
\it Orcid link: \href{0000-0002-8515-7108}{0000-0002-8515-7108}}



\begin{abstract}
In this paper, we study the eigenvalue problem
\[\left\{\begin{array}{cl}
-\hbox{div\,}\left(a(x)\frac{Du}{|Du|}\right)=\Lambda\, b(x)\frac{u}{|u|} & \text{in }\Omega\\
u=0 & \text{on }\partial\Omega,
\end{array}\right.\]
where $a(x)$ and $b(x)$ are suitable nonnegative functions.
We prove that the first eigenvalue coincides with the weighted Cheeger constant. To see this identity, we analyze the behavior of the first Dirichlet eigenvalue of the weighted $p$-Laplacian operator as $p$ goes to $1$. In the case that the weight $a(x)$ be Lipschitz-continuous, we show that the limit of eigenvalues of the weighted $p$-Laplacian exists, and it is the weighted Cheeger constant. In addition, we check that the sequence of normalized $p$-eigenfunctions converges to the normalized eigenfunction of our limiting problem, which turns out to be bounded. For more general weights, we identify the first eigenvalue of the weighted 1-Laplacian operator with the weighted Cheeger constant and prove that the associated eigenfunction is bounded.
\end{abstract}

\maketitle

\section{Introduction}
The study of properties of differential operators with weights is a topic of great interest in mathematical analysis. This interest especially concerns weighted $p$-Laplacian equations, including eigenvalue problems. In this paper, our objective is to check that the first eigenvalue of the weighted 1-Laplacian coincides with the weighted Cheeger constant.
We also consider the problem of the first eigenvalue of the weighted $p$-Laplacian and let $p$ go to 1. We aim to analyze the convergence of both eigenvalues and eigenfunctions. Moreover, we want to provide conditions to ensure that their limits form an eigenpair of the limit problem, which involves the weighted $1$-Laplacian.

\subsection{Our starting point}
The Dirichlet spectrum of the $p$-Laplacian operator $\Delta_p(u)=\dive\left(|\nabla u|^{p-2}\nabla u\right)$ was exhaustively analyzed at the end of the 20th century, beginning in the mid-80s (see \cite{dethelin1, dethelin2}). The main issues concern features of the first eigenvalue (it is the least eigenvalue and the only one whose eigenfunctions do not change of sign) which is proved to be  simple on bounded domains and isolated (see \cite{anane, lindqvist0}). Moreover, the existence of a sequence of eigenvalues that tends to infinity (see \cite{garcia-peral}) is shown. For further information on the $p$-Laplacian spectrum, we refer to \cite{le} and \cite{lindqvist}.

The analysis of the first eigenvalue has been extended to the following weighted eigenvalue problem
\begin{equation}
\label{elliptic_problem}
    \begin{cases}
        -\dive \left(a(x)\abs{\nabla u_p}^{p-2}\nabla u_p\right)=\lambda b(x)\abs{u_p}^{p-2}u_p &\textrm{in}\; \Omega\\
        u_p=0 &\textrm{on}\; \partial \Omega,
    \end{cases}
\end{equation}
where $\Omega$ is an open and bounded domain, while $a(x)$ is a positive function and $b(x)$ is a nonnegative one, possibly unbounded that does not vanish identically. A solution is sought in the weighted Sobolev space $W_0^{1,p}(\Omega, a)$. This problem was originally introduced in \cite{Drabek} and later studied in \cite{Cuesta, leschmit2, Drabek2, Drabek1} (related problems can be found in \cite{Gavit, Ho}). The hypotheses assumed in \cite[Chapter 3]{Drabek} are:

\begin{subequations}\label{eq:A-p}
\begin{equation}\label{eq:A-p1} 
a(x)\in L_{loc}^1(\Omega)
\end{equation}
\begin{equation}\label{eq:A-p2} 
a^{-1}(x)\in L^s(\Omega) \text{  for certain } s\in \left(\frac Np,\infty\right)\cap \left[\frac1{p-1},\infty\right)
\end{equation}
\begin{equation}\label{eq:A-p3} 
b(x)\in L^r(\Omega) \text{ for some } r>\frac{Ns}{ps-N}.
\end{equation}
\end{subequations}

Under these assumptions, it is proved that
\begin{equation}
\label{autoval}
\lambda_{p}=\inf_{v\in W_0^{1,p}(\Omega,a)}\left\{\int_\Omega a(x)|\nabla v|^pdx \>:\> \int_\Omega b(x)|v|^pdx=1\right\}
\end{equation}
is positive, and it is the least eigenvalue of problem \eqref{elliptic_problem}. This eigenvalue turns out to be simple, and its associated eigenfunction is bounded and does not change its sign.

Regarding the existence of a sequence of eigenvalues in \eqref{elliptic_problem} that tends to $\infty$, it can be found in \cite{leschmit2}.

It is worth pointing out that the assumption of non-negativeness of function $b(x)$ can be weakened. Indeed, weights $b(x)$ that change sign are handled in \cite{Cuesta}, under the sole hypothesis the positive part $b^+(x)$ does not vanish, and the least positive eigenvalue is studied (see also \cite{leschmit2} where the existence of a nondecreasing sequence of positive eigenvalues is established). This extension also applies in the limit case $p=1$. Nonetheless, in this paper, we are only dealing with non-negative weights.

\subsection{Limit problem}
As already mentioned, our main concern is to analyze the principal eigenvalue of the limit problem
\begin{equation}
    \label{limit_problem_intro}
    \begin{cases}
    -\dive \left(a(x)\frac{Du}{|Du|}\right)=\Lambda  b(x) \frac{u}{|u|} &\textrm{in}\; \Omega\\
    u=0 &\textrm{on}\; \partial \Omega.
\end{cases}
\end{equation}
As far as we know, this problem has not yet been addressed.

The limit operator of the $p$--Laplacian as $p$ tends to 1 is formally given by $\displaystyle \dive\left(\frac{Du}{|Du|}\right)$. It is a highly singular operator whose natural energy space is $BV(\Omega)$, the space of all functions of bounded variation. The notion of solution for the $1$-Laplacian with a Dirichlet boundary condition was introduced in \cite{mazon, demengel1} through a vector field which plays the r\^ole of the quotient $Du/|Du|$ (where $Du$ and $|Du|$ are Radon measures). It is also subject to a weak form for the boundary condition. To obtain the desired solution, a Green's formula is needed. It is stated in \cite{mazon} and is based on the pairing theory developed by Anzellotti in \cite{anzellotti}. 

After these pioneering articles, the study of problems involving $1$-Laplacian has grown hugely, mainly due to its applications to image processing. This interest includes the anisotropic case (see, for example, \cite{moll} for a bounded weight and \cite{chata} for an unbounded one).

The first eigenvalue for the $1$-Laplacian was first studied in \cite{demengel2} and \cite{kawohl1}. In the first paper, the infimum of the Rayleigh quotient
\[\lambda_1(\Omega)=\inf\left\{\frac{\int_\Omega |\nabla v|\, dx}{\int_\Omega |v|\, dx}\>:\> v\in W_0^{1,1}(\Omega)\backslash\{0\}\right\}\]
is also identified as
\begin{equation}\label{rayl1}
 \inf\left\{\frac{\int_\Omega |D v|+\int_{\partial\Omega}|v|\, d\mathcal H^{N-1}}{\int_\Omega |v|\, dx}\>:\> v\in BV(\Omega)\backslash\{0\}\right\}
\end{equation}
and the eigenfunctions are found to have the form of characteristic functions of Cacciopoli sets. The second paper is devoted to checking that the first eigenvalue is exactly the Cheeger constant of the domain while a corresponding eigenfunction is recognized as the characteristic function of a Cheeger set. For further developments, we refer to \cite{alter, demengel3, kawohl2}.

The notion of the spectrum of the $1$-Laplacian is due to Chang in \cite{chang} using critical point theory for non-smooth functionals: an eigenvalue is a critical value, in the sense of the weak slope, of the functional given by
$J(v)=\int_\Omega |D v|+\int_{\partial\Omega}|v|\, d\mathcal H^{N-1}$. As expected, every eigenvalue is a solution of the associated Dirichlet problem involving the $1$-Laplacian. Nevertheless, there are solutions to this elliptic problem that are not critical values of functional $J$, and Chang provides some examples of these spurious solutions.

In our weighted setting, the variational formulation \eqref{rayl1} becomes 
\begin{equation}\label{rayl2}
\inf\left\{\frac{\int_\Omega a(x)|D v|+\int_{\partial\Omega}a(x)|v|\, d\mathcal H^{N-1}}{\int_\Omega b(x)|v|\, dx}\>:\> v\in BV(\Omega, a)\backslash\{0\}\right\}.
\end{equation}
Our aim is to see that this infimum is attained and its value coincides with both the first eigenvalue $\Lambda(\Omega,a,b)$ and the weighted Cheeger's constant.

\subsection{Asymptotic behavior}
In order to study equations involving the $1$-Laplacian, a widely used method is to take the limit as $p$ goes to 1 of similar equations in which the $p$-Laplacian operator appears (see, for instance, \cite{mazon}). This is also the procedure employed in \cite{kawohl1} to study the principal Dirichlet eigenvalue of the $1$-Laplacian (see \cite{barbator, pietra} for the Robin boundary condition).

The behavior of the Dirichlet spectrum of the $p$-Laplacian as $p$ goes to $1$ is studied in \cite{chang}, for the $1$-dimensional case, and in \cite{littig}, for the general case (see also \cite{sabina} and \cite{kajikiya} for the radial case). In \cite{littig}, it is shown that eigenvalues of the $p$-Laplacian converge to the corresponding eigenvalue of the $1$-Laplacian, avoiding spurious solutions. Hence, a good strategy to obtain the spectrum of the $1$-Laplacian is to study the convergence of the spectrum of the $p$-Laplacian operator.

\subsection{Assumptions}
Since our goal is to study the convergence of eigenpairs of problem \eqref{elliptic_problem} to an eigenpair of \eqref{limit_problem_intro}, the two prerequisites are to define a reasonable weighted BV-space and to get an extension of Anzellotti's theory in which unbounded weights can be handled. 

The first issue has been analyzed in \cite{baldi}. Nevertheless, we need a slightly different theory, which we only develop what is strictly necessary for our purposes. Regarding the second one, although we already have the results of \cite{CdCS}, we do not need such a large extension. A comprehensive definition of the pairing is all that is required to obtain a Green’s formula where every term, except for the pairing term itself, involves functions rather than Radon measures or abstract distributions.

We are able to adapt both prerequisites in Section \ref{preliminari} by assuming that our weight belongs to $W^{1,1}(\Omega)$.

On the other hand, our hypotheses on functions $a(x)$ and $b(x)$ must be compatible with those in \eqref{eq:A-p}. Thus, we have to let $s$ go to $\infty$, which leads to $a^{-1}(x)\in L^{\infty}(\Omega)$ and $b(x)\in L^r(\Omega)$, with $r>N$.

\subsection{Plan of the paper}
To summarize, the structure of the paper is the following. In Section \ref{preliminari}, we review some preliminary results as well as the properties of weights and weighted spaces. We also obtain a weighted Green formula. In Section \ref{weighted_plaplaciano} we prove the existence of the solution to the problem \eqref{elliptic_problem} together with some properties involving the first eigenfunction and the first eigenvalue. In Section \ref{problema_limite_sezione} we analyze the convergence process of the weighted $p$-Laplacian problem and check that, assuming the existence of the limit of $\lambda_p$ as $p$ goes to $1$, this limit is an eigenvalue of the weighted 1-Laplacian problem. In Section \ref{CheegerSec} we prove the existence of the limit and identify it with the weighted Cheeger constant for Lipschitz weights. Moreover, we also prove that for arbitrary weights the least eigenvalue of the 1-Laplacian is precisely the weighted Cheeger constant.

\section{Preliminaries}
\label{preliminari}
\subsection{Notation}
\begin{itemize}
\item $\Omega$ is a bounded open set in $\R^N$ ($N\geq 2$), with Lipschitz boundary $\partial \Omega$;

\item $\nu$ is the outer unit normal to the boundary of $\Omega$;

\item $\mathcal H^{N-1}(E)$ will denote the $N-1$-dimensional Hausdorff measure of a set $E$;

\item $\abs{E}$ denotes the Lebesgue measure of $E$;

\item $C_c^1(\Omega)$ denotes for the space of functions with compact support which are continuously differentiable on $\Omega$;

\item $C_c^\infty(\Omega)$ stands for the space of all functions with compact support having derivatives of all orders;

\item $T_k(s)=\min\{|s|,k\}\sign(s)$ denotes the truncation function at levels $\pm k$.
\end{itemize}
\subsection{Assumptions on weights}

Henceford, we consider that the weight $a(x)$ exhibits the following properties:

\begin{subequations}\label{eq:A}
\begin{equation}\label{A1}
a(x)\in W^{1,1}(\Omega).
\end{equation}
\begin{equation}\label{A2}
\text{There exists } \mu>0 \text{ such that } \mu\le a(x) \text{ a.e. in } \Omega.
\end{equation}
\begin{equation}\label{A3}
b(x)\in L^r(\Omega) \text{ for certain }r>N.
\end{equation}
\end{subequations}

\smallskip

We stress that the first hypothesis implies that the weight  $a(x)$ can be taken as $\mathcal H^{N-1}$-a.e. determined. More precisely, there exists a precise representative of $a(x)$ which is well-defined $\mathcal H^{N-1}$-a.e. In what follows, we denote by $a(x)$ this precise representative. Furthermore, each weight has a trace on $\partial\Omega$ belonging to $L^1(\partial\Omega)$ (see \cite{AFP} or \cite{EG}).

\begin{rem}\label{ipo-p}
We explicitly point out that each weight $a(x)$ satisfy
\begin{enumerate}
\item $a(x)\in L^1(\Omega)$
\item $a^{-1}(x)\in L^s(\Omega)$ for certain $s\in\Big(\frac Np, \infty\Big)\cap\Big[\frac1{p-1}, \infty\Big)$
\item $b\in L^r(\Omega)$ for some $r>\frac{Ns}{ps-N}$
\end{enumerate}
As a consequence, for every $p>1$,  it verifies all the conditions in \cite{Drabek} and so the theory developed in this book applies.
\end{rem}

\subsection{Functional setting}
 To analyze this kind of problems, we have to define the weighted Sobolev space $W^{1,p}(\Omega,a)$ and the weighted BV space. For properties of weighted Sobolev spaces, we refer to \cite{Drabek}; a different theory can be found in \cite{HKM}.
\begin{defn}
\label{sobolev_space}
We define the Sobolev space $W^{1,p}(\Omega, a)$ to be the completion of $C^{\infty}(\R^N)$ with respect to the norm 
$$\norm{w}_{W^{1,p}(\Omega,a)}=\left(\int_{\Omega}\abs{w}^p\;dx+\int_{\Omega}a(x)\abs{\nabla w}^p\;dx\right)^{\frac{1}{p}},$$
while $W_0^{1,p}(\Omega, a)$ is the completion of $C_c^{\infty}(\Omega)$ with respect to the same norm.
\end{defn}

We recall that, given $w\in L^1(\Omega)$, its total variation is defined as
\begin{equation}\label{varia}
\text{var\,}w=\sup \left\{\int_{\Omega}w \dive \varphi\, dx ; \varphi \in C^1_c(\Omega, \R^N), \;\textrm {s.t}\; \abs{\varphi(x)}\leq 1\right\},
\end{equation}
and denoted by $\int_\Omega|Dw|$. The space $BV(\Omega)$ is made up for all those $L^1$-functions having a finite total variation. The total variation defines a Radon measure whose main properties can be found in \cite{AFP, EG}. We highlight that $|Dw|$ is absolutely continuous with respect to the measure $\mathcal H^{N-1}$ (see \cite[Lemma 3.76]{AFP}). As a consequence, every weight $a(x)\in W^{1,1}(\Omega)$ is $|Dw|$-measurable. Another important feature of $|Dw|$ is that there exists a $|Dw|$-measurable function $\sigma\>:\>\Omega\to\R^N$ satisfying
\begin{align}
   \label{var-1}  &|\sigma(x)|=1 \quad |Dw|-a.e.  \\
   \label{var-2} &\int_\Omega w \dive \varphi\, dx =-\int_\Omega\langle \varphi,\sigma\rangle\, |Dw|\quad \forall \varphi\in C_c^1(\Omega; \R^N).
\end{align}

We will next discuss the weighted case.
 Let $a(x)$ be a weight function satisfying \eqref{eq:A} and let $w\in L^1(\Omega)$. Following \cite{baldi}, we define the weighted total variation by
 $$\text{var}_{a}w=\sup \left\{\int_{\Omega}w \dive \varphi\, dx ; \varphi \in C^1_c(\Omega, \R^N), \;\textrm {s.t}\; \abs{\varphi(x)}\leq a(x)\right\}.$$

 \begin{defn}
 Given a weight $a(x)$, the weighted $BV$ space is defined as (see \cite{baldi})
 $$BV(\Omega, a)=\left\{w \in L^1(\Omega): \textrm{var}_{a} w < +\infty\right\}.$$
 \end{defn}

 The space $BV(\Omega, a)$ is a Banach space when endowed with the norm
 $$\norm{w}_{BV(\Omega,a)}=\text{var}_{a} w+\int_{\Omega}\abs{w}\;dx.$$

Owing to \eqref{A2}, we deduce $\mu\, \text{var\,} w\le \text{var}_{a}w$, so that the continuous embedding $BV(\Omega, a)\hookrightarrow BV(\Omega)$ holds.
As a consequence
  $$BV(\Omega,a) \hookrightarrow L^1(\partial\Omega)$$
  $$BV(\Omega,a) \hookrightarrow L^{\frac{N}{N-1}}(\Omega)$$
  and, in addition, this embedding is compact into $L^q(\Omega)$ for $1\le q<\frac{N}{N-1}$.
 
 We point out that for every $w\in BV(\Omega, a)$, we may take the supremum for all smooth functions $|\varphi(x)|\le a(x)$ on both sides of \eqref{var-2} to get
 $$\text{var}_{a}w=\int_\Omega a(x)|Dw|$$
 (see \cite{baldi}).

A property we need to use is the lower semicontinuity of the weighted total variation. Since for each $\varphi \in C^1_c(\Omega, \R^N)$ the functional $w\mapsto \int_{\Omega}w \dive \varphi\, dx$ is continuous with respect to the $L^1$-convergence, it follows that the functional 
$$w\mapsto \int_\Omega a(x)|Dw|\qquad w \in BV(\Omega, a)$$ 
is lower semicontinuous with respect to this convergence.
This fact is improved in the next result using standard techniques.

\begin{thm}[Lower semicontinuity]
    \label{lower_semicontinuity}
    The functional defined by
$$w\mapsto \int_{\Omega}a(x)\abs{D w}+\int_{\partial\Omega}a(x)\abs{w}\, \mathcal H^{N-1}\qquad w \in BV(\Omega, a)$$ is lower semicontinuous with respect to the $L^1$-convergence.
\end{thm}
\begin{proof}
Consider an open set $\Omega'$ satisfying $\Omega' \supset \supset \Omega$.
Let $w\in BV(\Omega,a)$ be fixed. 
If we define 
$$\overline{w}(x)=
\begin{cases}
   w(x) &\textrm{if}\;x\in \Omega\\
   0 &\textrm{if}\; x\in \Omega'\setminus\Omega\,,
\end{cases}
    $$
then $\overline w\in BV(\Omega')$ and (see \cite[Corollary 3.89]{AFP}) its gradient is 
$$D\overline w=Dw-w\nu \mathcal H^{N-1}\res{\partial\Omega}\,.$$
Hence, 
$$|D\overline w|=|Dw|+|w|\mathcal H^{N-1}\res{\partial\Omega}$$
and the weighted total variation of $w$ is given by
$$\int_{\Omega'}a(x)|D\overline w|=\int_\Omega a(x)|Dw|+\int_{\partial\Omega}a(x)|w|\, d\mathcal H^{N-1}.$$
The result follows since we already know that the total variation is lower semicontinuous with respect to the $L^1$-convergence.
\end{proof}

\subsection{Weighted vector fields and Green's formula}

This Section is devoted to recall and extend the Anzellotti theory. Let $z \in L^{\infty}(\Omega, \R^N)$ be such that $\dive z \in L^1(\Omega)$. For any $u \in BV(\Omega)\cap L^\infty(\Omega)$,  Anzellotti  in \cite{anzellotti} introduced the distribution $(z, Du):C_c^{\infty}(\Omega)\rightarrow\R$ defined as
$$\langle (z, Du),\varphi\rangle =-\int_{\Omega}u\varphi \dive z\, dx-\int_{\Omega}uz \cdot \nabla \varphi\, dx, \quad \forall \varphi \in C_{c}^{\infty}(\Omega).$$ Then it is proved that $(z, Du)$ is actually a Radon measure with finite total variation and
\begin{equation}\label{disu}
\abs{(z,Du)}\leq \norm{z}_{L^{\infty}(\Omega;\R^N)}\abs{Du},
\end{equation}
holds as measures. We notice that this inequality implies that if $a(x)$ is a weight and $u\in BV(\Omega, a)$, then (the precise representative of) $a(x)$ is a $|(z,Du)|$-summable function.  Furthermore, Anzellotti defines the weak trace of the normal component of $z$ as an application $[z, \nu]; \partial \Omega \rightarrow \R$ which satisfies $[z, \nu] \in L^{\infty}(\partial \Omega)$ and $\norm{[z, \nu]}_{L^{\infty}(\partial \Omega)}\leq \norm{z}_{L^{\infty}(\Omega, \R^N)}$. Finally, a Green's formula involving $[z, \nu]$ and $(z, Du)$ holds, namely:
\begin{equation}
\label{green_classica}
\int_{\Omega}(z, Du)+\int_{\Omega}u \dive z\;dx=\int_{\partial \Omega}u[z,\nu]\;d\mathcal{H}^{N-1}.
\end{equation}

Next, we want to adapt this theory to our case. To this end, let $z\in L^\infty(\Omega; \R^N)$ satisfy $\dive (a(x) z)\in L^1(\Omega)$.

Firstly, observe that 
$$\dive (a(x) z)=a(x) \dive z+z\cdot \nabla a(x)$$
in the sense of distributions.
Recalling that $a(x)\in W^{1,1}(\Omega)$, it follows from $\dive (a(x) z)\in L^1(\Omega)$ that $a(x) \dive z\in L^1(\Omega)$. Thus, \eqref{A2} implies $ \dive z\in L^1(\Omega)$, so that Anzellotti's theory applies. 

Observe that if $ \dive z\in L^1(\Omega)$, $u\in BV(\Omega)\cap L^\infty(\Omega)$ and $w\in W^{1,1}(\Omega)\cap L^\infty(\Omega)$, then (the precise representative of) $w$ is summable with respect to the measure $(z,Du)$, so that $w(z,Du)$ is a finite measure. On the other hand, if $z\in L^\infty(\Omega; \R^N)$ satisfies $\dive z\in L^1(\Omega)$ and $w\in W^{1,1}(\Omega)\cap L^\infty(\Omega)$, then $wz\in L^\infty(\Omega; \R^N)$ and 
\[\dive(w z)=w\dive z+z\cdot\nabla w\in L^1(\Omega);\]
thus the pairing $(wz, Du)$ is well-defined for every $u\in BV(\Omega)\cap L^\infty(\Omega)$.

\begin{lem}
\label{sacar}
Let $z\in L^\infty(\Omega; \R^N)$ satisfy $\dive z\in L^1(\Omega)$.
 If $u\in BV(\Omega)\cap L^\infty(\Omega)$ and $w\in W^{1,1}(\Omega)\cap L^\infty(\Omega)$, then 
 $$(wz, Du)=w(z,Du)\qquad\text{as measures.}$$
 \end{lem}

\begin{proof}
We may follow the proof of \cite[Proposition 2.3]{MS} and arrive at \cite[Equation (2.17)]{MS}, which in our context reads as
$$(wz,Du)+u z\cdot\nabla w=u z\cdot\nabla w+w(z,Du)\,.$$
Simplifying, we are done.

\end{proof}

\begin{prop}
 \label{green-1}
    Let $z\in L^\infty(\Omega; \R^N)$ satisfy $\dive (a(x) z)\in L^1(\Omega)$  and let $u \in BV(\Omega, a)\cap L^\infty(\Omega)$. Then the following Green formula holds true:
    $$\int_\Omega u\dive (a(x) z)\, dx+\int_\Omega a(x) (z,Du)=\int_{\partial\Omega} u a(x) [z,\nu]\, d\mathcal H^{N-1}.$$
 \end{prop}
 
 \begin{proof}
 Take $k>\mu$ and notice that
 $$\dive (T_k(a(x)) z)=T_k(a(x))\dive z+z\cdot \nabla T_k(a(x))$$
 wherewith $\dive (T_k(a(x)) z)\in L^1(\Omega)$ and so we may use Anzellotti's theory.

Now, Green's formula leads to
\begin{multline*}
 \int_\Omega uT_k(a(x))\dive z\, dx+\int_\Omega uz\cdot \nabla T_k(a(x))\, dx=\int_\Omega u\dive (T_k(a(x))z)\, dx\\
 =-\int_\Omega (T_k(a(x)) z, Du)+\int_{\partial\Omega}u[T_k(a(x))z, \nu]\, d\mathcal H^{N-1}.
\end{multline*}
We manipulate the right-hand side, applying Lemma \ref{sacar} in the first term and \cite[Lemma 5.6]{caselles} in the second. So, we deduce
\begin{multline}\label{eq:fun1}
\int_\Omega uT_k(a(x))\dive z\, dx+\int_\Omega uz\cdot \nabla T_k(a(x))\, dx\\
=-\int_\Omega T_k(a(x)) (z, Du)+\int_{\partial\Omega}u T_k(a(x)) [z, \nu]\, d\mathcal H^{N-1}.
\end{multline}
Our aim is to let $k$ go to infinity in \eqref{eq:fun1} and it is a consequence of Lebesgue's Theorem. To check it, just observe that functions $u$, $z$ and $[z,\nu]$ are bounded, $a(x)\dive z\in L^1(\Omega)$, $a(x)\in W^{1,1}(\Omega)$ and $a(x)$ is summable with respect to the measure $|(z,Du)|$ (owing to $u\in BV(\Omega, a)$). Therefore, \eqref{eq:fun1} becomes
$$
\int_\Omega u a(x) \dive z\, dx+\int_\Omega uz\cdot \nabla a(x)\, dx
=-\int_\Omega a(x) (z, Du)+\int_{\partial\Omega}u a(x) [z, \nu]\, d\mathcal H^{N-1}.
$$
The proof finishes taking into account that $\dive(a(x)z)=a(x) \dive z+z\cdot \nabla a(x)$.
 \end{proof}

 \begin{prop}\label{green-2}
    Let $z\in L^\infty(\Omega; \R^N)$ satisfy $\dive (a(x) z)\in L^N(\Omega)$  and let $u \in BV(\Omega, a)$. Assume also that $a(x)|u|\in L^1(\partial\Omega)$.
    Then the following Green formula holds true:
    $$\int_\Omega u\dive (a(x) z)\, dx+\int_\Omega a(x) (z,Du)=\int_{\partial\Omega} u a(x) [z,\nu]\, d\mathcal H^{N-1}.$$
 \end{prop}

 \begin{proof}
     We may apply Proposition \ref{green-1} to every truncation $T_k(u)$, obtaining
     $$\int_\Omega T_k(u)\dive (a(x) z)\, dx+\int_\Omega a(x) (z,DT_k(u))=\int_{\partial\Omega} T_k(u) a(x) [z,\nu]\, d\mathcal H^{N-1}.$$
     We are allowed to let $k$ go to infinity due to Lebesgue's Theorem. Indeed, the first term is a consequence of $\dive (a(x) z)\in L^N(\Omega)$  and $u \in BV(\Omega, a)\subset L^{\frac{N}{N-1}}(\Omega)$, in the second one the integrand is dominated by $\|z\|_\infty a(x) |Du|$ (a finite measure) while in the third one is dominated by $a(x)|u|\in L^1(\partial\Omega)$. Passing to the limit we get the desired identity.
 \end{proof}

 We point out that we have not defined the pairing $(a(x)z, Du)$. Instead, we are checking that the measure $a(x)(z, Du)$ verifies what the definition of $(a(x)z, Du)$ should be. We note that this pairing can be studied following the theory in \cite{CdCS}.

 \begin{prop}\label{no-def}
  Let $z\in L^\infty(\Omega; \R^N)$ satisfy $\dive (a(x) z)\in L^1(\Omega)$  and let $u \in BV(\Omega, a)\cap L^\infty(\Omega)$. Then
  $$\int_\Omega\varphi a(x) (z,Du)=-\int_\Omega \varphi  u \dive (a(x)z)\, dx-\int_\Omega u a(x) z\cdot\nabla \varphi\, dx$$
  for all $\varphi\in C_c^\infty(\Omega)$.
 \end{prop}

 \begin{proof}
     By Lemma \ref{sacar}
     $$\int_\Omega\varphi T_k(a(x)) (z,Du)=-\int_\Omega \varphi  u \dive (T_k(a(x))z)\, dx-\int_\Omega  u T_k(a(x)) z\cdot\nabla \varphi\, dx$$
     holds for all $k>0$.
     The result follows by taking the limit when $k$ goes to $\infty$ and it can be done as in the proof of Proposition \ref{green-1}.
 \end{proof}

  \section{Weighted $p$-Laplacian problems}
\label{weighted_plaplaciano}
  \subsection{Existence and main features}
  \begin{defn}
We say that $\lambda$ is an eigenvalue of \eqref{elliptic_problem} if there exists a nontrivial $u_p \in W_0^{1,p}(\Omega, a)$ which is a weak solution to \eqref{elliptic_problem}, that is,
\begin{equation}
\label{weak_formula}
    \int_{\Omega}a(x)\abs{\nabla u_p}^{p-2}\nabla u_p\cdot \nabla \varphi \:dx=\lambda\int_{\Omega}b(x)\abs{u_p}^{p-2}u_p\varphi \: dx.
\end{equation}
for every $\varphi\in  W_0^{1,p}(\Omega,a)$.
  \end{defn}

In \cite{Drabek} the following existence result is proved.
 \begin{prop}
     There exists the least eigenvalue $\lambda_p(\Omega,a,b)>0$ whose corresponding eigenfunction $u_p \in W^{1,p}(\Omega,a)$ is nonnegative.
     This eigenvalue is given by \eqref{autoval} and $u_p$ is a minimizer.
 \end{prop}
 
 Now, we recall some properties of the first eigenvalue and the associated eigenfunction. The the standard case is addressed in \cite{anane} and \cite{lindqvist0}, while the weighted problem is studied in \cite{Drabek}.
 \begin{thm}
     Let $u_p$ be an eigenfunction associated to $\lambda_p(\Omega,a,b)$.  Then the following conditions hold:     
     \begin{enumerate}
         \item $u_p\in L^{\infty}(\Omega)$,
         \item $\lambda_p(\Omega,a,b)$ is simple, which means, there is just a nonnegative eigenfunction satisfying $\int_\Omega b(x)|u_p|^p\, dx=1$.
     \end{enumerate}
     
 \end{thm}

 We remark that $u_p$ might vanish inside $\Omega$.
  
  \subsection{The limit of eigenvalues as $p$ goes to 1}
 Next, we want to analyze the behavior of $\lambda_p(\Omega,a,b)$ as $p \to 1$. Nevertheless, we are not able to directly show that there exists $\lim_{p\to1}\lambda_p(\Omega,a,b)$ since the arguments employed in the isotropic setting do not hold. Instead, we are considering a sequence of exponents $\{p_m\}$ such that $p_m>1$, $p_m\to1$ and $\lim_{m\to\infty}\lambda_{p_m}=\Lambda$ exists. Our subsequent concern will be to study a condition which guarantees that this limit does not depend on the chosen sequence (see Subsection \ref{Sec-5.2}). The first step is to make sure that the family $\{\lambda_p\}$ is bounded; it is due to the following technical result.
 \begin{lem}
 \label{boundedness}
     There exist $p_0\in(1,2)$ and $C>0$ fulfilling 
     $0\le \lambda_p\le C$ for all $1<p<p_0$.
 \end{lem}
 \begin{proof}
 Choose $u_0\in W_0^{1,2}(\Omega, a)$ and recall that the continuous imbedding $W_0^{1,2}(\Omega, a)\hookrightarrow W_0^{1,p}(\Omega, a)$ holds for all $1\le p\le 2$. Now fix $p_0\in(1,2)$ satisfying 
 \[(p_0-1) \int_\Omega b(x)\, dx\le \frac12  \int_\Omega b(x)|u_0|\, dx\,.\]
 Then, for all $1<p<p_0$, apply Young's inequality and get
 \begin{multline*}
 \frac{\displaystyle \int_\Omega a(x)|\nabla u_0|^pdx}{\displaystyle \int_\Omega b(x)|u_0|^pdx} \le 
  \frac{\displaystyle \frac p2\int_\Omega a(x)|\nabla u_0|^2dx+\frac{2-p}{p}\int_\Omega a(x)\, dx}{\displaystyle p\int_\Omega b(x)|u_0|\, dx-(p-1)\int_\Omega b(x)\, dx}\\[2mm]
  \le  \frac{\displaystyle \int_\Omega a(x)|\nabla u_0|^2dx+\int_\Omega a(x)\, dx}{\displaystyle \frac12\int_\Omega b(x)|u_0|\,dx}.
 \end{multline*}
 Hence,
 \[\lambda_p\le \frac{\displaystyle \int_\Omega a(x)|\nabla u_0|^2dx+\int_\Omega a(x)\, dx}{\displaystyle \frac12\int_\Omega b(x)|u_0|\,dx}\]
   and we are done.
 \end{proof}
 
 In what follows we normalize the $p$-eigenfunctions assuming that 
 $$\int_{\Omega}b(x)\abs{u_p}^{p}\;dx=1\,,\qquad 1<p<p_0\,.$$
 The next step is to get a $BV$-estimate for these $p$-eigenfunctions $u_p$.  

  \begin{prop}\label{estimation}
  There exists $M>0$ satisfying
\begin{equation}\label{BV-bound}
\int_{\Omega}a(x)\abs{\nabla u_p}^{p}\;dx \leq M\,,\qquad 1<p<p_0\,.
\end{equation} 
  \end{prop}
   \begin{proof} 
   The  thesis follows from 
   $$\int_{\Omega}a(x)\abs{\nabla u_p}^{p}dx=\lambda_p\int_{\Omega}b(x)\abs{u_p}^pdx=\lambda_p,$$
   owing to Lemma \ref{boundedness}.
    \end{proof} 

 We point out that the arguments developed in the isotropic case (see \cite{lindqvist2}) allow us to check the existence of $\lim_{p\to1}\lambda_p$ under the strong assumption $a(x)\le \mathcal C b(x)$ for some constant $\mathcal C>0$. This fact provides evidence that it is reasonable to think that it should be true in the general framework.
 
\begin{prop}\label{truco}
Assume there exists $\mathcal C>0$ such that $a(x)\le \mathcal C b(x)$. 

If $1<p<s$, then $p \, \mathcal{C}^{-1/p}\lambda_p^{\frac{1}{p}}\leq s \, \mathcal{C}^{-1/s}\lambda_s^{\frac{1}{s}}$.
 \end{prop}
 \begin{proof}
Let $\psi: W^{1,s}(\Omega, a) \rightarrow W^{1,p}(\Omega,a)$ be defined as $$u\in W^{1,s}(\Omega, a)\rightarrow \abs{u}^{\frac{s}{p}-1}u\in W^{1,p}(\Omega,a).$$ 
     Denoting
     $$\mathcal{M}_s=\left\{u\in W^{1,s}(\Omega,a):\int_{\Omega}b(x)\abs{u}^s=1\right\},$$
   we have that $u \in \mathcal{M}_s$ implies $\psi (u) \in \mathcal{M}_p$, indeed if $u\in \mathcal{M}_s$, we have that 
   $$\int_{\Omega}b(x)\abs{\psi(u)}^p\;dx=\int_{\Omega}b(x)\abs{u}^s\;dx=1$$
   and
   \begin{align*}  \int_{\Omega}a(x)\abs{\nabla(\psi(u))}^p\;dx&=\int_{\Omega}a(x)\left(\dfrac{s}{p}\right)^{p}\abs{u}^{\left(\frac{s}{p}-1\right)} \abs{\nabla u}^pdx\\& =\left(\dfrac{s}{p}\right)^{p}\int_{\Omega}a(x)\abs{u}^{s-p}\abs{\nabla u}^pdx\\&\leq \left(\dfrac{s}{p}\right)^{p}\left(\int_{\Omega}a(x)\abs{u}^{s}dx\right)^{\frac{s-p}{s}}\left(\int_{\Omega}a(x)\abs{\nabla u}^sdx\right)^{\frac{p}{s}}.  
   \end{align*}

   Using the definition of the first eigenvalue, we obtain
   \begin{align*}
   \lambda_p^{\frac{1}{p}}=\inf_{w \in \mathcal{M}_p}\left(\int_{\Omega}a(x)\abs{\nabla w}^p\right)^{\frac{1}{p}}&\leq \inf_{u\in \mathcal{M}_s}\left(\int_{\Omega}a(x)\abs{\nabla \psi(u)}^s\right)^{\frac{1}{s}}\\&\leq \inf_{u\in \mathcal{M}_s}\dfrac{s}{p}\left(\int_{\Omega}a(x)\abs{\nabla u}^s\right)^{\frac{1}{s}} \mathcal{C}^{\frac{1}{p}-\frac{1}{s}}\\&=\frac{s}{p}\mathcal{C}^{\frac{1}{p}-\frac{1}{s}}\lambda_s^{1/s}
   \end{align*}
   and the claim follows.
\end{proof}

\section{Limit Problem}
\label{problema_limite_sezione}
\subsection{Main features}
The objective of this Section is to analyze the limit problem that involves the weighted $1$-Laplacian
\begin{equation}
\label{limit_problem_intro2}
\begin{cases}
    -\dive \left(a(x)\frac{Du}{|Du|}\right)=\Lambda b(x) \frac{u}{|u|} &\textrm{in}\; \Omega\\
    u=0 &\textrm{on}\; \partial \Omega.
  
\end{cases}
  \end{equation}
 
   \begin{defn}
  \label{sollimite}
  We say that $\Lambda$ is an eigenvalue of \eqref{limit_problem_intro2} if there exists a nontrivial $u \in BV(\Omega, a)$ which is a solution to \eqref{limit_problem_intro2} in the sense that there exist $z\in L^{\infty}(\Omega,\R^N)$ and $\gamma \in L^\infty(\Omega)$  such that $\norm{z}_{\infty}\leq 1$, $\gamma \in sign(u)$ and the following conditions hold
  \begin{enumerate}
      \item $-\dive(a(x)z)=\Lambda b(x) \gamma$ in the sense of distributions;
      \item $(z,Du)=\abs{Du}$ as measures;
      \item $-[z,\nu]\in \sign(u)$, $\mathcal H^{N-1}$-a.e. on $\partial \Omega$.
  \end{enumerate}
  In this case, $u$ is said to be an eigenfunction of \eqref{limit_problem_intro2} associated to $\Lambda$.
  \end{defn}
  
  Observe that $\Lambda b(x) \gamma\in L^r(\Omega)$, with $r>N$, so that $\dive(a(x)z)\in L^N(\Omega)$. Hence, to apply Green's formula, we require either $v\in BV(\Omega)\cap L^\infty(\Omega)$ (Proposition \ref{green-1}), or $a(x)|v|\in L^1(\partial\Omega)$ (Proposition \ref{green-2}). If $v$ satisfies one of these assumptions, then the first condition of Definition \ref{sollimite} yields 
  \[\int_\Omega a(x) (z,Dv)-\int_{\partial \Omega}v[z,\nu]\, d\mathcal H^{N-1}=\Lambda \int_{\Omega}b(x)\gamma v\, dx.\]
  It is not obvious that we may take the solution $u$ as $v$ in the above formula. 
  That is why it is so convenient the following result.

\begin{prop}\label{acot}
      Every solution to \eqref{limit_problem_intro2} is bounded and satisfies
      \[\|u\|_\infty\le \left(\frac{S|\Lambda|\|b\|_r}{\mu} \right)^\frac{Nr}{r-N}\|u\|_1\,,\]
      where $S$ stands for the Sobolev constant of $BV(\Omega,a) \hookrightarrow L^{\frac{N}{N-1}}(\Omega)$.
  \end{prop}
  \begin{proof}
  To begin with, we denote $A_k=\{|u|>k\}$ for $k>0$. 
      Let us consider the auxiliary function defined as $G_k(s)=(\abs{s}-k)^+\textrm{sign} (s)$, taking $T_1(G_k(u))$ as test function in the first condition of Definition \ref{sollimite}, applying Green's function and employing condition (3), it yields
      \begin{multline*}
          \int_\Omega a(x)(z,DT_1(G_k(u)))+\int_{\partial\Omega}a(x)|T_1(G_k(u))|\, d\mathcal H^{N-1}\\
          \le |\Lambda| \int_\Omega b(x) |T_1(G_k(u))|\, dx.
      \end{multline*}
      Furthermore, condition (2) and \cite[Proposition 2.8]{anzellotti} imply that $(z,DT_1(G_k(u)))=|DT_1(G_k(u))|$ as measures. Hence,
      \begin{multline*}
          \int_\Omega a(x)|DT_1(G_k(u))|+\int_{\partial\Omega}a(x)|T_1(G_k(u))|\, d\mathcal H^{N-1}\\
          \le |\Lambda| \int_\Omega b(x) |T_1(G_k(u))|\, dx\\
          \le |\Lambda|\left(\int_\Omega b(x)^r\right)^{\frac1r}\left(\int_{\Omega}|T_1(G_k(u))|^{\frac{N}{N-1}}\right)^{\frac{N-1}{N}}|A_k|^{\frac{r-N}{Nr}}
      \end{multline*}
      by H\"older's inequality. Now it follows from Sobolev's inequality that
      \begin{multline*}
      \mu\left(\int_\Omega |T_1(G_k(u))|^{\frac{N}{N-1}}\right)^{\frac{N-1}{N}}\\
      \le S\left[ \int_\Omega a(x)|DT_1(G_k(u))|+\int_{\partial\Omega}a(x)|T_1(G_k(u))|\, d\mathcal H^{N-1}\right]\,.
      \end{multline*}
      Gathering the above inequalities, we obtain
      \[\left(\int_\Omega |T_1(G_k(u))|^{\frac{N}{N-1}}\right)^{\frac{N-1}{N}}\le \frac{S |\Lambda| \|b\|_r}{\mu} |A_k|^{\frac{r-N}{Nr}}\left(\int_{\Omega}|T_1(G_k(u))|^{\frac{N}{N-1}}\right)^{\frac{N-1}{N}}\,.\]
      Choose now $k$ such that $k>\left(\frac{S|\Lambda|\|b\|_r}{\mu} \right)^{\frac{Nr}{r-N}}\|u\|_1$, wherewith we have
      \begin{equation}\label{acota1}
      \frac{S|\Lambda| \| b\|_r}{\mu}|A_k|^{\frac{r-N}{Nr}}\le \frac{S|\Lambda| \| b\|_r}{\mu}\frac{\|u\|_1^{\frac{r-N}{Nr}}}{k^{\frac{r-N}{Nr}}}<1
      \end{equation}
      Hence, \eqref{acota1} implies $\int_\Omega|T_1(G_k(u))|^{\frac{N}{N-1}}=0$, so that $|u|\le k$. Taking into account that it holds for every $k>\left(\frac{S|\Lambda|\|b\|_r}{\mu} \right)^{\frac{Nr}{r-N}}\|u\|_1$, we are done.      
  \end{proof}

   \begin{rem}\label{test-sol}
 We point out that, thanks to Proposition \ref{acot}, the solution $u$ can be taken as test function in condition (1) of Definition \ref{sollimite} and applying Green's formula of Proposition \ref{green-1}, we get
 \[\int_\Omega a(x)(z,Du)-\int_{\partial\Omega}a(x)u[z,\nu]\, d\mathcal H^{N-1}=\Lambda\int_\Omega b(x)\gamma u\, dx.\]
 Now our hipotheses on $z$, $[z,\nu]$ and $\gamma$ lead to 
 \begin{equation}\label{eq:test-sol}
     \int_\Omega a(x)|Du|+\int_{\partial\Omega}a(x)|u|\, d\mathcal H^{N-1}=\Lambda\int_\Omega b(x)|u|\, dx.
 \end{equation}
 Several simple consequences of this identity are in order:
 \begin{enumerate}
      \item Every solution $u$ to problem \eqref{limit_problem_intro2} satisfies 
 \[\int_\Omega b(x)|u|\, dx\ne0.\]
 Indeed, recalling assumption $a(x)\ge\mu>0$, we deduce from \eqref{eq:test-sol} that
 \[0<\int_\Omega |Du|+\int_{\partial\Omega}|u|\, d\mathcal H^{N-1}\le\frac{\Lambda}{\mu}\int_\Omega b(x)|u|\, dx\]
 from where our claims follow.
     \item Every eigenvalue $\Lambda$ is positive
     \item If there exists the minimum of
     \[\frac{\int_\Omega a(x)|Du|+\int_{\partial\Omega}a(x)|u|\, d\mathcal H^{N-1}}{\int_\Omega b(x)|u|\, dx},\]
     taken over all $u\in BV(\Omega, a)$ such that $ \int_\Omega b(x)|u|\, dx\ne0$, 
     this value is the first eigenvalue of \eqref{sollimite}

 \end{enumerate}
   \end{rem}

 \subsection{Existence}
 Henceforth $\{p_m\}$ stands for a sequence satisfying $1<p_m<p_0$ and $p_m\to1$.  Moreover, we denote  $\lambda_m=\lambda_{p_m}$ and $u_m=u_{p_m}$, 
and we require that the sequence $\{\lambda_m\}$ converges to a certain value $\Lambda$. We are also assuming that 
 $$\int_{\Omega}b(x)\abs{u_m}^{p_m}dx=1\,,\qquad m\in\N\,.$$

 \begin{prop}\label{limit1}
      There exists a non-negative $u\in BV(\Omega,a)$ such that, up to a subsequence, 
$$u_m \rightarrow u \qquad \textrm{in}\; L^q(\Omega)\,,\quad for\ 1\le q<\frac{N}{N-1}.$$
  \end{prop}
  \begin{proof}
      It is enough to apply the compact embedding. Indeed, if we take $u_m$ as test function in \eqref{weak_formula} we have that 

$$\int_{\Omega}a(x) \abs{\nabla u_m}^{p_m}\; dx =\int_{\Omega}\lambda_m b(x) \abs{u_m}^{p_m}\;dx.$$
Applying Young's inequality, we have that 
\begin{align*}
    \int_{\Omega}a(x)\abs{\nabla u_m}\;dx&\leq \dfrac{1}{p_m}\int_{\Omega}a(x)\abs{\nabla u_m}^{p_m}\;dx +\dfrac{p_m-1}{p_m}\left(\int_{\Omega}a(x)\right)\\&=\dfrac{M}{p_m}+\dfrac{p_m-1}{p_m}\int_{\Omega}a(x)\;dx,
\end{align*}
where Proposition \ref{estimation} has been used. 
 Hence, 
 $$\mu\left[\int_\Omega|\nabla u_m|\, dx+\int_{\partial\Omega}|u_m|\, d\mathcal H^{N-1}\right]\le \int_{\Omega}a(x)\abs{\nabla u_m}\;dx\le M+\int_{\Omega}a(x)\;dx$$
 and we have that $\{u_m\}$ is bounded in $W^{1,1}(\Omega)$ and then, up to subsequences, $u_m$ converges to a function $u$ strongly in $L^q(\Omega)$ for $1\le q<\frac{N}{N-1}$ and almost everywhere in $\Omega$. The pointwise convergence yields $u\ge0$. For the lower semicontinuity of the weighted total variation, we get $u \in BV (\Omega, a)$.
  \end{proof}
We are proving that this limit can be seen as an eigenfunction of the limit problem \eqref{sollimite}.

\begin{prop}\label{limit2}
    The limit 
    \[b(x)|u_m|^{p_m}\to b(x)|u|\]
    holds strongly in $L^1(\Omega)$.
\end{prop}
\begin{proof}
    Since $b\in L^r(\Omega)$ for some $1\le r'<\frac{N}{N-1}$, it follows that $\frac{N(r-1)}{r(N-1)}>1$. Fix $1<q<\frac{N(r-1)}{r(N-1)}$ and consider $1<p_m<q$. We first have 
    \[b(x)|u_m|^{p_m}\le b(x) (1+|u_m|)^q\]
    and the left hand side converges strongly in $L^1(\Omega)$ owing to $\frac{qr}{r-1}<\frac{N}{N-1}$, so that the left hand side is equi-integrable. Furthermore, the pointwise convergence $b(x)|u_m|^{p_m}\to b(x)|u|$ and Vitali's Theorem leads to the $L^1$--convergence. 
\end{proof}

\begin{cor}\label{limit3}
 The identity
      \[\int_\Omega b(x)|u|\, dx=1\]
      is verified and so $u$ is nontrivial.
  \end{cor}
  
  \begin{thm}
The following convergences hold true
      \begin{enumerate}
    \item There exists a vector field $z \in L^{\infty}(\Omega, \R^N)$, with $\norm{z}_{\infty}\leq 1$ such that, up to subsequences, 
    \begin{equation}\label{conv-vector1}
    |\nabla u_m|^{p_m-2}\nabla u_m \rightharpoonup  z
    \qquad\hbox{weakly in } L^s(\Omega, \R^N)
    \end{equation}
 for all $1\le s<\infty$.
    \item There exists $\gamma \in L^{\infty}(\Omega)$, with $\norm{\gamma}_{\infty}\leq 1$ such that, up to subsequences, 
    \begin{equation}\label{conv-vector2}
    |u_m|^{p_m-2} u_m \rightharpoonup \gamma
    \qquad\hbox{weakly in } L^s(\Omega)
    \end{equation}
 for all $1\le s<\infty$.
    \end{enumerate}
  \end{thm}
  \begin{proof} 
  1) We want to prove the existence of the vector field $z$, then we proceed as in \cite[Theorem 3.5]{Mercaldo}. Fix $s\ge1$. For every $p$ such that $s<p'$, applying \eqref{BV-bound} and the H\"older inequality, it yields
\begin{multline*}\label{def-z}
\mu\int_{\Omega}\left(\abs{\nabla u_m}^{p_m-1}\right)^s\;dx \leq \int_{\Omega}a(x)\left(\abs{\nabla u_m}^{p_m-1}\right)^s\;dx\\ \leq \left(\int_{\Omega}a(x)\abs{\nabla u_m}^{p_m}\;dx\right)^{\frac{(p_m-1)s}{p_m}}\left(\int_{\Omega}a(x) \;dx\right)^{1-\frac{(p_m-1)s}{p_m}}\\
\le M^{\frac{(p_m-1)s}{p_m}}\left(\int_{\Omega}a(x)\; dx\right)^{1-\frac{(p_m-1)s}{p_m}}
\end{multline*}
Thus, $\abs{\nabla u_m}^{p_m-2}\nabla u_m$  is bounded in both $L^s(\Omega)$ and $L^s(\Omega, a)$. Following \cite{mazon, Mercaldo} and up to a subsequence, we get $z \in L^{\infty}(\Omega, \R^N)$ such that $\norm{z}_{\infty}\leq 1$ and $\abs{\nabla u_m}^{p_m-2}\nabla u_m \rightharpoonup z$ weakly in $L^s(\Omega, \R^N)$ for all $s<\infty$.
On the other hand, for each $s<\infty$ there exists $y_s \in L^s(\Omega,a)^N $ such that 
\begin{equation}\label{debil1}
\abs{\nabla u_m}^{p_m-2}\nabla u_m \rightharpoonup y_s\qquad \hbox{weakly in } L^s(\Omega, a)^N\,.
\end{equation}

We shall prove that the two limits are equal, that is: $z=y_s$. To this aim, we consider $a(x) \chi_{\{a(x)<k\}}\varphi$, with $\varphi \in L^{\infty}(\Omega,\R^N)$, and obtain 
$$\lim_{m\rightarrow\infty}\int_{\{a(x)<k\}}a(x)\abs{\nabla u_m}^{p_m-2}\nabla u_m \cdot \varphi = \int_{\{a(x)<k\}}a(x)z\cdot \varphi.$$
Now, having in mind \eqref{debil1}, consider the function $\chi_{\{a(x)<k\}}\varphi$ to get
$$\int_{\{a(x)<k\}}a(x)\abs{\nabla u_m}^{p_m-2}\nabla u_m \cdot \varphi \xrightarrow{m\rightarrow \infty} \int_{\{a(x)<k\}}a(x)y_s\cdot \varphi.$$
Hence, $$\int_{\{a(x)<k\}}a(x)z\cdot\varphi=\int_{\{a(x)<k\}}a(x) y_s\cdot \varphi\; \forall k.$$ 
Since this holds for every $\varphi \in L^{\infty}(\Omega, \R^N)$, it implies $a(x)\chi_{\{a<k\}} z =a(x)\chi_{\{a<k\}} y_s$ for all $k>0$. Then $z=y_s$, for all $s$, as desired.\\
 2) A similar procedure can be employed to prove the second claim. Indeed,
\begin{equation*}
\int_{\Omega}\left(\abs{ u_m}^{p_m-1}\right)^s dx \leq \left(\int_{\Omega}\abs{ u_m}^{p_m} dx\right)^{\frac{(p_m-1)s}{p_m}} |\Omega|^{1-\frac{(p_m-1)s}{p_m}},
\end{equation*}
then there exists $\gamma\in L^s(\Omega, b)$ such that $\abs{u_m}^{p_m-2}u_m$ weakly converges to $\gamma\in L^s(\Omega, b)$. By a diagonal procedure, the limit is independent of $s$.
Then,
$$\left[\int_{\Omega}|\gamma|^sdx\right]^{\frac1s}\leq |\Omega|^{\frac{1}{s}},$$
and passing to the limit for $s\rightarrow +\infty$, using \cite{pietra2}, we have that $\gamma\in L^{\infty}(\Omega)$ and $\norm{\gamma}_{L^{\infty}(\Omega)}\leq1$. 

Furthermore, if $x\in \Omega$ is such that $u(x)> 0$, then by the $L^1$-convergence of $u_m$ to $u$,
$$u_m(x)^{p_m-1}\rightarrow 1 \qquad \textrm{for}\; m\rightarrow \infty.$$
Indeed, $u_m(x)^{p_m-1}\rightharpoonup  \gamma$ weakly in $L^1(\Omega)$ and $u_m\to u$ strongly in $L^1(\Omega)$ and a.e. Hence, in $\{u(x)>0\}$ we have
\[\liminf_{m\to\infty}e^{(p_m-1)\log u_m(x)}\le \gamma(x)\le \limsup_{m\to\infty}e^{(p_m-1)\log u_m(x)}.\]
Being $$\lim_{m\rightarrow \infty}(p_m-1)\log u_m(x)=0,$$ we deduce
$$\liminf_{m\to\infty}e^{(p_m-1)\log u_m(x)}=1 \qquad \textrm{and}\qquad \limsup_{m\to\infty}e^{(p_m-1)\log u_m(x)}=1,$$  we can conclude that $\gamma(x)=1$ in $\{u(x)>0\}$. Therefore, $\gamma\in\sign(u)$.
\end{proof}

 \begin{thm}\label{main0}
 Function  $u$ is a solution to the eigenvalue problem \eqref{limit_problem_intro2}
 \end{thm}
 \begin{proof}
It remains to check the three conditions of Definition \ref{sollimite}.

\bigskip

Let $\varphi\in C^{\infty}_c(\Omega)$ be taken as test function in \eqref{weak_formula}, then 
$$
\int_{\Omega}\lambda_pb(x)\abs{u_m}^{p_m-2}u_m\varphi\;dx=\int_{\Omega}a(x)\abs{\nabla u_m}^{p_m-2}\nabla u_m \cdot \nabla \varphi\;dx. 
$$
Passing to the limit for $m$, we have that
$$\int_{\Omega}\Lambda b(x) \gamma\varphi\;dx=\int_{\Omega}a(x)z\cdot\nabla\varphi\;dx,$$
which implies that $-\dive(a(x)z)=\Lambda b(x) \gamma$ holds, in the sense of distributions.

\bigskip

The next step is to prove that $a(x)(z, Du)=a(x)\abs{Du}$. To do so, we take as test function $\varphi T_k(u_m)$, with $\varphi \in C^1_c(\Omega),\varphi \geq 0$. It yields 
\begin{multline*}
\int_{\Omega}a(x)\abs{\nabla T_k(u_m)}^{p_m}\varphi\;dx+\int_{\Omega}a(x)\abs{\nabla u_m}^{p_m-2}\nabla u_m\cdot\nabla \varphi T_k(u_m)\;dx\\=\int_{\Omega}\lambda_m b(x)\abs{u_m}^{p_m-2}u_m T_k(u_m)\varphi\;dx,
\end{multline*}
and then
\begin{multline*}
\int_{\Omega}a(x)\abs{\nabla T_k(u_m)}^{p_m}\varphi\;dx\\=-\int_{\Omega}a(x)\abs{\nabla u_m}^{p_m-2}\nabla u_m \cdot \nabla \varphi T_k(u_m)\;dx\\+\int_{\Omega}\lambda_m b(x)\abs{u_m}^{p_m-2}u_m T_k(u_m)\varphi\;dx.
\end{multline*}
Applying Young's inequality,  the lower semicontinuity (as $m$ goes to $\infty$) yields 
$$\int_{\Omega}\varphi a(x)\abs{D T_k(u)}\;dx\leq-\int_{\Omega}T_k(u)a(x)z\nabla \varphi\;dx+\Lambda \int_{\Omega}b(x) \gamma T_k(u)\varphi\;dx.$$
Thanks to Proposition \ref{no-def}, we have
$$\int_{\Omega}\varphi a(x)\abs{D T_k(u)}\;dx\leq
\langle a(x)(z,DT_k(u)),\varphi\rangle$$
for every nonnegative $\varphi\in C_c^\infty(\Omega)$.
Hence, $a(x)\abs{D T_k(u)}\le a(x)(z,DT_k(u))$ as measures and assumption \eqref{A2} yields $\abs{D T_k(u)}\le (z,DT_k(u))$.
The equality follows from \eqref{disu}.
Notice that we may let $k$ go to infinity obtaining $\abs{D u}= (z,Du)$ as measures.

\bigskip

For the last point, we choose $u_m$ as test function in \eqref{weak_formula} and applying the Young inequality, we have that 
\begin{multline*}
\int_{\Omega}a(x)\abs{\nabla u_m}\;dx+\int_{\partial \Omega}a(x)\abs{u_m}\;d\mathcal H^{N-1}\\
\leq \dfrac{1}{p_m}\int_{\Omega}a(x)\abs{\nabla u_m}^{p_m}\;dx+\dfrac{p_m-1}{p_m}\int_{\Omega}a(x)\;dx\\
=\dfrac{\lambda_m}{p_m}\int_{\Omega}b(x)\abs{u_m}^{p_m}\;dx+\dfrac{p_m-1}{p_m}\int_{\Omega}a(x)\;dx.
\end{multline*}
Since the left-hand side is lower semicontinuous, as proved in Theorem \eqref{lower_semicontinuity}, we get 
\begin{equation}\label{var1}
\int_{\Omega}a(x)\abs{Du}+\int_{\partial \Omega}a(x)\abs{u}\;d\mathcal{H}^{N-1}\leq \Lambda\int_{\Omega}b(x)\abs{u}\;dx\,.
\end{equation}
A straightforward consequence is $a(x)u\in L^1(\partial\Omega)$.

On the other hand, the previous step yields
$$\int_{\Omega}a(x)\abs{Du}+\int_{\partial \Omega}a(x)\abs{u}\;d\mathcal{H}^{N-1}\leq -\int_{\Omega}u\dive(a(x)z)\;dx$$
and so Green's formula of Proposition \ref{green-2} gives
\begin{equation*}
\int_{\Omega}a(x)\abs{Du}+\int_{\partial \Omega}a(x)\abs{u}\;d\mathcal{H}^{N-1}\leq\int_{\Omega}a(x)(z,Du)-\int_{\partial \Omega}u a(x) [z,\nu]\;d\mathcal{H}^{N-1}.
\end{equation*}
Simplifying, we arrive at
$$\int_{\partial \Omega}a(x)\abs{u}+a(x)u[z,\nu]\, d\mathcal H^{N-1} \leq 0.$$ 
Finally, observing that the integrand is nonnegative, this inequality implies that the integrand vanishes: $a(x)\abs{u}+a(x)u[z,\nu]=0$ and so $-[z,\nu] \in \sign (u)$.
\end{proof}

We point out that it also follows from \eqref{var1} that 
\begin{equation}\label{des-Chee0}
\int_{\Omega}a(x)\abs{Du}+\int_{\partial \Omega}a(x)\abs{u}\;d\mathcal{H}^{N-1}\leq \Lambda\,.\end{equation}
Thus,
\begin{equation}\label{des-Chee}
\inf\left\{\int_\Omega a(x)|Dv|+\int_{\partial\Omega}a(x)|v|\, d\mathcal H^{N-1}\>:\> \int_\Omega b(x)|v|\, dx=1\right\}\le \Lambda,
\end{equation}
where the infimum is taken over all $v\in BV(\Omega, a)$.

 \begin{rem}\label{compar}
 It is worth comparing the eigenvalues corresponding to different weights $a_1(x)$ and $a_2(x)$. Fixed $p>1$,  if $a_1(x)\le a_2(x)$, then $\lambda_p(\Omega, a_1, b)\le \lambda_p(\Omega, a_2, b)$. To check this inequality consider $v_i\in W_0^{1,p}(\Omega, a_i)$ which minimizes 
 \[\inf_{v\in W_0^{1,p}(\Omega,a_i)}\left\{\int_\Omega a_i(x)|\nabla v|^pdx \>:\> \int_\Omega b(x)|v|^pdx=1\right\}\,,\quad i=1,2\,.\]
 Since $v_2\in W_0^{1,p}(\Omega,a_2)\subset W_0^{1,p}(\Omega,a_1)$, it follows that 
 \begin{multline*}
 \lambda_p(\Omega, a_1, b)= \int_\Omega a_1(x)|\nabla v_1|^pdx\le 
 \int_\Omega a_1(x)|\nabla v_2|^pdx\\
 \le \int_\Omega a_2(x)|\nabla v_2|^pdx= \lambda_p(\Omega, a_2, b)\,.
 \end{multline*}

 Going further, we obtain a comparison between the limits corresponding to different weights $a_1(x)$ and $a_2(x)$. Indeed, if $\{p_m\}$ is a sequence of real numbers that converges to 1 and such that there exist 
 $$\Lambda_i=\lim_{m\to\infty}\lambda_{p_m}(\Omega, a_i, b),\quad i=1,2\,,$$ then 
 $\Lambda_1\le \Lambda_2$.
\end{rem}

\section{The weighted Cheeger problem}
\label{CheegerSec}

\subsection{The weighted Cheeger constants}
In this Section, we point out the link between our sequential limit $\Lambda$ and the weighted Cheeger constant. This fact will allow us to conclude that $\lim_{p\to1}\lambda_p$ actually exists (under a suitable assumption) and the weighted Cheeger constant is, also in our framework, the first eigenvalue. 

The analysis of the Cheeger problem along with uniqueness and regularity properties of Cheeger sets has been studied by many authors (see \cite{parini}). We next consider its weighted counterpart.\\
The definition of the weighted perimeter of a set $E\subset \Omega$ was introduced in \cite{Ion-Lachand} (see also  \cite{BBF} and \cite{CFM}), 
$$P_a(E,\Omega)=\sup\left\{ \int_E \dive \varphi \; dx: \varphi \in C^1_c(\Omega;\R^N), \abs{\varphi}\leq a(x) \right\}.$$
Observe that for every $E\subset \Omega$ of finite perimeter it holds
$$P_a(E, \Omega)=\int_\Omega a(x)|D\chi_{E}|$$
and
$$P_a(E, \R^N)=\int_\Omega a(x)|D\chi_{E}|+\int_{\partial\Omega}a(x)\chi_E\, d\mathcal H^{N-1}.$$

\begin{defn}We define the weighted Cheeger constant as
$$h(\Omega, a, b)=\inf_E \frac{P_a(E,\R^N)}{\int_E b(x)\, dx}$$
where the infimum is taken over all measurable subsets $E\subset\Omega$ having finite weighted perimeter and satisfying $\int_E b(x)\, dx\ne0$.
\end{defn}

In our weighted setting a co-area formula is satisfied (see \cite[Theorem 2.2]{cianchi_fusco}). For nonnegative functions it reads as follows.

\begin{thm}
  Let  $u\in BV(\Omega, a)$ be nonnegative, then
  $$\int_\Omega a(x)|Du|+\int_{\partial\Omega}a(x) u\, d\mathcal H^{N-1}=\int_{0}^{+\infty} P_a(\{u>t\})\, dt,$$
  here $P_a(\{u>t\})$ means $P_a(\{u>t\},\R^N)$.
\end{thm}

\begin{prop}\label{prop:minim}
    Consider the minimization problem
    \begin{equation}\label{first-chee}
    \overline h:=\inf_{w}\dfrac{\int_{\Omega}a(x)\abs{D w}+\int_{\partial\Omega}a(x)|w|\, d\mathcal H^{N-1}}{\int_{\Omega}b(x)\abs{w}\, dx}
    \end{equation}
    with $w$ varying over all $w\in BV(\Omega, a)$ such that $\int_{\Omega}b(x)\abs{w}\, dx\ne0$.
    This infimum is achieved and  $\overline h=h(\Omega, a, b)$.
\end{prop}
\begin{proof}
    Observe that we may minimize the functional 
    $$w\mapsto\int_{\Omega}a(x)\abs{D w}+\int_{\partial\Omega}a(x)|w|\, d\mathcal H^{N-1}$$
    under the condition $\int_{\Omega}b(x)\abs{w}\, dx=1$. Let $w_n$ be a minimizing sequence satisfying
    $$\int_{\Omega}a(x)\abs{D w_n}+\int_{\partial\Omega}a(x)|w_n|\, d\mathcal H^{N-1}\le \overline h+\frac1n$$
    and
    $$\int_\Omega b(x)|w_n|\, dx=1\,.$$
    Then
    $$\int_{\Omega}a(x)\abs{D w_n}+\int_{\partial\Omega}a(x)|w_n|\, d\mathcal H^{N-1}\le\overline h+1 =:H$$
    so that
    $$\mu \left[\int_{\Omega}\abs{D w_n}+\int_{\partial\Omega}|w_n|\, d\mathcal H^{N-1}\right]\le H$$
    wherewith the sequence $w_n$ is bounded in $BV(\Omega)$.
    Hence, there exists $w\in L^1(\Omega)$ and a subsequence (still denoted by $w_n$) such that $w_n$ converges to $w$ strongly in $L^q(\Omega)$ for all $1\le q<\frac{N}{N-1}$. It follows from the lower semicontinuity that
    \begin{multline*}
    \int_{\Omega}a(x)\abs{D w}+\int_{\partial\Omega}a(x)|w|\, d\mathcal H^{N-1}\\\le\liminf_{n\to\infty}\int_{\Omega}a(x)\abs{D w_n}+\int_{\partial\Omega}a(x)|w_n|\, d\mathcal H^{N-1}\le \overline h\,.
    \end{multline*}
    So $w\in BV(\Omega, a)$ and $\int_{\Omega}a(x)\abs{D w}+\int_{\partial\Omega}a(x)|w|\, d\mathcal H^{N-1}=\overline h$.  

    On the other hand, $b\in L^r(\Omega)$, with $r>N$, and the strong convergence of $w_n$ to $w$ in $L^{r'}(\Omega)$ implies 
    $$\int_\Omega b(x)|w|\, dx=\lim_{n\to\infty}\int_\Omega b(x) |w_n|\, dx=1\,,$$
    and we deduce that $w$ minimizes our problem. 
    
    It follows from the chain rule (see \cite[Theorem 3.96]{AFP}) that $\big|D |w|\big|\le \abs{D w}$, so that 
    \[\int_{\Omega}a(x) \big|D |w|\big|+\int_{\partial\Omega}a(x)|w|\, d\mathcal H^{N-1}\le\overline h\]
     and $\int_\Omega b(x)|w|\, dx=1$. Hence, $|w|$ is a minimizer as well and so we may assume that the minimizer is nonnegative.
    
    Next, we will check the identity $\overline h=h(\Omega, a, b)$, first notice that $\chi_E\in BV(\Omega, a)$ for every $E\subset\Omega$ with finite weighted perimeter; thus $\overline h\le h(\Omega, a, b)$.

    To see the opposite inequality, let $w\in BV(\Omega, a)$ be a nonnegative minimizer of our problem. Then the weighted coarea formula implies
    \begin{multline*}
        0=\int_\Omega a(x)|Dw|+\int_{\partial\Omega}a(x) w\, d\mathcal H^{N-1}-\overline h \int_\Omega b(x) w\, dx\\
        \ge \int_\Omega a(x)|Dw|+\int_{\partial\Omega}a(x) w\, d\mathcal H^{N-1}-h(\Omega, a, b) \int_\Omega b(x) w\, dx\\
        =\int_{0}^{+\infty}\left(P_a(\{w>t\})-h(\Omega, a, b)\int_{\{w>t\}}b(x)\, dx\right)\, dt\ge0\,.
    \end{multline*}
    Therefore, $\overline h=h(\Omega, a, b)$ and we are done. It is worth remarking that the integrand in the last line vanishes for almost all $t\in \R$, so that the level sets $\{w>t\}$ are weighted Cheeger sets for a.e. $t\in\R$.
\end{proof}

\begin{cor}\label{comp-Chee}
Every eigenvalue $\Lambda$ of \eqref{limit_problem_intro2} satisfies
\[h(\Omega, a, b)\le\Lambda\,.\]
\end{cor}

\begin{proof}
Let $u$ be a solution to \eqref{limit_problem_intro2}. 
By Remark \ref{test-sol}, we know that $\int_\Omega b(x)|u|\, dx\ne0$, so that we may choose $u$ such that $\int_\Omega b(x)|u|\, dx=1$. Taking $u$ as a test function and applying Green's formula, we obtain
\[h(\Omega, a, b)\le \int_\Omega a(x)|Du|+\int_{\partial\Omega}a(x)|u|\, d\mathcal H^{N-1}=\Lambda\]
as desired.
\end{proof}

We have to consider another related constant, that is equal to the Cheeger constant in the isotropic case.

\begin{defn} The smooth weighted Cheeger constant is defined as
$$\sigma(\Omega, a, b)=\inf_D \frac{P_a(D, \Omega)}{\int_D b(x)\, dx}$$
where the infimum is taken over all open subsets $D\subset\Omega$ having $C^\infty$ boundary, such that $\overline D\subset\Omega$ and $\int_D b(x)\, dx\ne0$.
\end{defn}

\begin{rem}\label{oss}
It is straightforward that $h(\Omega, a, b)\le \sigma(\Omega, a, b)$ since every open set $D$ having $C^\infty$ boundary has finite weighted perimeter.
\end{rem}

\begin{prop}\label{minim}
 Consider the minimization problem
    $$\overline \sigma:=\inf_{v}\dfrac{\int_{\Omega}a(x)\abs{\nabla v}\, dx}{\int_{\Omega}b(x)\abs{v}\, dx}$$
    where $v$ is taken over all $v\in C_c^\infty(\Omega)$ such that $\int_{\Omega}b(x)\abs{v}\, dx\ne0$.
    
    Then $\overline \sigma=\sigma(\Omega, a, b)$.
\end{prop}

\begin{proof}
To begin with, we point out that 
 $$\overline \sigma=\inf_{w}\dfrac{\int_{\Omega}a(x)\abs{\nabla w}\, dx}{\int_{\Omega}b(x)\abs{w}\, dx}$$
    here $w$ varies over all $w\in W_0^{1,1}(\Omega, a)$ such that $\int_{\Omega}b(x)\abs{w}\, dx\ne0$.
     This fact is a consequence of the density $\overline{C_c^\infty(\Omega)}=W_0^{1,1}(\Omega, a)$. 

We are next proving $\overline \sigma\le \sigma(\Omega, a, b)$.
Given an open set with smooth boundary $D$ such that $\overline D\subset\Omega$ and $\int_{D}b(x)\, dx\ne0$, consider the $\varepsilon$-neighborhood 
$I_{\varepsilon}(D)=\{x\in\Omega : dist(x, D)<\varepsilon\}$, where $\varepsilon$ is small enough to satisfy $\overline{I_{\varepsilon}(D)}\subset\Omega$. Then consider a function given by
 $w=1$ for all $x\in \overline D$, $w(x)=0$ for all $x\in\Omega\setminus I_{\varepsilon}(D)$ and $|\nabla w|=\frac1\varepsilon$ in $I_\varepsilon(D)\setminus D$. Obviously, $w\in W_0^{1,1}(\Omega, a)$ and
\[\int_\Omega b(x)|w|\, dx\ge\int_D b(x)\, dx>0.\]
On the other hand, we also have
\[\int_{I_\varepsilon(D)\backslash D}a(x)\, dx
    =\int_{\partial D}a(x) \varepsilon\, d\mathcal H^{N-1}+\varepsilon o(\varepsilon).\]
Hence,
\begin{equation*}
\overline\sigma\le \dfrac{\int_\Omega a(x)|\nabla w|\, dx}{\int_\Omega b(x)|w|\, dx}
\le \dfrac{\int_{I_\varepsilon(D)\setminus D} a(x)\frac1\varepsilon\, dx}{\int_D b(x)\, dx}
\le \dfrac{\int_{\partial D} a(x)\, dx+o(\varepsilon)}{\int_D b(x)\, dx},
\end{equation*}
so that 
\[\overline\sigma\le \dfrac{\int_{\partial D} a(x)\, dx}{\int_D b(x)\, dx}=\dfrac{P_a(D,\Omega)}{\int_D b(x)\, dx}.\]
Therefore, $\overline \sigma\le \sigma(\Omega, a, b)$.

To check the last inequality, choose $\delta>0$ and take $v\in C_c^\infty(\Omega)$ such that $\int_\Omega b(x)|v|\, dx\ne0$ and
\[ \dfrac{\int_\Omega a(x) |\nabla v|\, dx}{\int_\Omega b(x)|v|\, dx}\le \overline\sigma+\delta\,.\]
As usual, the homogeneity of the quotient on the left hand side allows us to assume  $\int_\Omega b(x)|v|\, dx=1$. Moreover, since $\big|\nabla |v|\big|=|\nabla v|$, taking $|v|$ instead of $v$, we may assume that $v\in W_0^{1,1}(\Omega, a)$ is nonnegative. Then, owing to the coarea formula, 
\begin{multline*}
\delta\ge \int_\Omega a(x)|\nabla v|\, dx-\overline\sigma
\ge \int_\Omega a(x)|\nabla v|\, dx-\sigma(\Omega, a, b)\\
=\int_{0}^{+\infty}\left(P_a(\{v>t\})-\sigma(\Omega, a, b)\int_{\{v>t\}}b(x)\, dx\right)\, dt\ge0\,.
\end{multline*}
Thus, the following inequalities hold:
\[\sigma(\Omega, a, b)\le \int_\Omega a(x) |\nabla v|\, dx\le \overline\sigma+\delta\,.\]
The arbitrariness of $\delta>0$ gives the inequality 
$\sigma(\Omega, a, b)\le \overline\sigma$.
\end{proof}

\begin{rem}\label{compar2}
Continuing with our comparison of weights, let us see what happens with the corresponding Cheeger constants. If $a_1(x)\le a_2(x)$, we deduce that $BV(\Omega, a_2(x))\subset BV(\Omega, a_1(x))$, so that the infimum taken over $BV(\Omega, a_2(x))$ is bigger than that taken over $BV(\Omega, a_1(x))$. Hence, $$h(\Omega, a_1, b)\le h(\Omega, a_2, b)\,.$$

The same argument serves to prove that if  $a_1(x)\le a_2(x)$, then 
$$\sigma (\Omega, a_1, b)\le \sigma (\Omega, a_2, b).$$
\end{rem}

\subsection{Convergence of $p$-eigenvalues and Cheeger's constants.}\label{Sec-5.2}

This subsection is devoted to show the connexion between the limits of the first eigenvalues of problem \eqref{elliptic_problem} and the Cheeger constants that have been introduced in the previous subsection. 

\begin{thm}\label{prop:iden1}
The following inequalities hold:
\[h(\Omega, a, b)\le \liminf_{p\to1^+}\lambda_p(\Omega, a, b)\] 
\[\limsup_{p\to1^+}\lambda_p(\Omega, a, b)\le \sigma(\Omega, a, b)\,.\]
\end{thm}

\begin{proof}
Consider a sequence $\{p_m\}$ such that $1<p_m<p_0$ and $p_m\to1$. We assume that the sequence $\{\lambda_m\}$ converges and $\Lambda=\lim_{m\to\infty}\lambda_{m}$. We have to check the inequalities 
\[h(\Omega, a, b)\le \Lambda\le \sigma(\Omega, a, b)\,.\]

Note that the inequality $h(\Omega, a, b)\le \Lambda$ follows from \eqref{des-Chee}. To check the remaining inequality, we proceed as in \cite[Corollary 6]{kawohl1}. Having in mind Proposition \ref{minim}, we fix $n\in\N$ and choose a smooth domain $D_n$ such that $\overline{D_n} \subset \Omega$, $\int_{D_n}b(x)\, dx\ne0$ and 
$$\frac{P_a(D_n,\Omega)}{\int_{D_n} b(x)\, dx}-\sigma(\Omega,a,b)\leq \frac{1}{n}\,.$$ 
Then, we approximate $\chi_{D_n}$ with the function defined as
\begin{equation}
\label{function_v}
v(x)=\begin{cases}
    1 & \textrm{on}\; D_n\\
    0 & \textrm{on}\; \Omega\setminus I_{\varepsilon}(D_n),
\end{cases}
\end{equation}
where $I_{\varepsilon}(D_n)$ is a $\varepsilon$-neighborhood of $D_n$. Furthermore, we take $\abs{\nabla v}=\frac{1}{\varepsilon}$ on an $\varepsilon$-layer outside $D_n$. In this way, for small positive $\varepsilon$, it yields $v\in W_0^{1,\infty}(\Omega, a)$. Now observe that
\begin{equation}\label{eq:deno}
\int_\Omega b(x)|v|^{p_m}dx\ge \int_{D_n} b(x)\, dx\,.
\end{equation}
On the other hand, it holds
\begin{multline}\label{eq:num}
    \int_\Omega a(x)|\nabla v|^{p_m}dx=\int_{I_\varepsilon(D_n)\backslash D_n}a(x)\frac1{\varepsilon^{p_m}}\, dx\\
    =\frac{1}{\varepsilon^{p_m}}\left[\int_{\partial D_n}a(x) \varepsilon\, d\mathcal H^{N-1}+\varepsilon \, o(\varepsilon)\right]
    =\varepsilon^{1-{p_m}}[P_a(D_n,\Omega)+o(\varepsilon)]\,.
\end{multline}
Gathering \eqref{eq:deno} and \eqref{eq:num}, we obtain
\[
\frac{\int_\Omega a(x)|\nabla v|^{p_m}dx}{\int_\Omega b(x)|v|^{p_m}dx}\le\varepsilon^{1-{p_m}}\frac{P_a(D_n,\Omega)+o(\varepsilon)}{\int_{D_n} b(x)\, dx}\,,
\]
and consequently
\begin{equation}\label{eq:esen}
\lambda_m\le \varepsilon^{1-{p_m}}\frac{P_a(D_n,\Omega)+o(\varepsilon)}{\int_{D_n} b(x)\, dx}\,.
\end{equation}
Letting $m$ go to $\infty$, we get
\[
\Lambda\le \frac{P_a(D_n,\Omega)+o(\varepsilon)}{\int_{D_n} b(x)\, dx}\le \sigma(\Omega, a, b)+o(\varepsilon)+\frac1n\,.
\]
Passing now to the limit, first for $\varepsilon\to0$ and then for $n\rightarrow \infty$, we conclude that
\[\Lambda\le \sigma(\Omega, a, b)\]
as desired. 
\end{proof}

Our next concern is to find a suitable class of weights for which the equality $h(\Omega, a, b)= \sigma(\Omega, a, b)$ holds.

\subsection{Existence of $\lim_{p\to1}\lambda_p(\Omega, a,b)$. The Lipschitz case.}

In this subsection, we assume that $a(x)$ is a Lipschitz function such that $a(x)\ge \mu>0$. We begin by considering a result inspired by \cite[Lemma 5.5]{anzellotti}. It can be proven using the Gagliardo's argument for identifying the trace of $W^{1,1}(\Omega)$.

\begin{lem}\label{gag}
For every $u\in BV(\Omega, a)$ and every $\varepsilon>0$ there exists a function $w\in W^{1,1}(\Omega, a)\cap C(\Omega)$ such that
\begin{enumerate}
\item $w\big|_{\{dist(\partial\Omega, x)>\varepsilon\}}=0$
\item $\|w\|_{r'}<\epsilon$
\item $w\big|_{\partial\Omega}=u\big|_{\partial\Omega}$
\item $\displaystyle \int_\Omega a(x) |\nabla w|\, dx\le \int_{\partial\Omega}a(x)|u|\, d\mathcal H^{N-1}+\varepsilon$
\end{enumerate}
\end{lem}

The following result is a consequence of the previous Lemma and \cite[Theorem 3.4.]{baldi}.

\begin{thm}\label{ugua}
Every Lipschitz weight $a(x)$ satisfies $$h(\Omega, a, b)= \sigma(\Omega, a, b).$$
\end{thm}

\begin{proof}
Recalling Remark \ref{oss}, we only have to check the inequality\\
$\sigma(\Omega, a, b)\le h(\Omega, a, b)$.

Let $u\in BV(\Omega, a)$ and apply Lemma \ref{gag} to get  $w\in W^{1,1}(\Omega, a)\cap C(\Omega)$ satisfying the four conditions. 

Then $u-w\in BV(\Omega, a)$, $(u-w)\big|_{\partial\Omega}=0$ and
\begin{multline*}
\int_\Omega a(x)|D(u-w)|\le \int_\Omega a(x)|Du|+\int_\Omega a(x)|\nabla w|\, dx\\
\le  \int_\Omega a(x)|Du|+\int_{\partial\Omega}a(x)|u|\, d\mathcal H^{N-1}+\varepsilon.
\end{multline*}
By \cite[Theorem 3.4.]{baldi}, there exists a sequence $v_n\in C^\infty(\Omega)\cap W^{1,1}(\Omega, a)$ such that $v_n\big|_{\partial\Omega}=0$,  and
\[\lim_{n\to\infty}\int_\Omega a(x)|\nabla v_n|\, dx=\int_\Omega a(x)|D(u-w)|.\]
We point out that $v_n\in W_0^{1,1}(\Omega)$.
Since $\{v_n\}$ is bounded in $ BV(\Omega, a)$, it is also bounded in $ BV(\Omega)$. Thus, up to subsequences, $\{v_n\}$ converges to $u-w$ in $ L^{r'}(\Omega)$ and so 
\[\lim_{n\to\infty}\int_\Omega b(x)|v_n|\, dx=\int_\Omega b(x)|u-w|\, dx\,.\]
Choosing $n$ large enough, we obtain
\begin{multline}\label{num}
\int_\Omega a(x)|\nabla v_n|\, dx\le \int_\Omega a(x)|D(u-w)|+\varepsilon\\
\le \int_\Omega a(x)|Du|+\int_{\partial\Omega}a(x)|u|\, d\mathcal H^{N-1}+2\varepsilon
\end{multline}
and
\begin{multline}\label{denom}
\int_\Omega b(x)|v_n|\, dx\ge \int_\Omega b(x)|u-w|\, dx-\varepsilon\\
\ge \int_\Omega b(x)|u|\, dx-\int_\Omega b(x)|w|\, dx-\varepsilon\\
\ge \int_\Omega b(x)|u|\, dx-\|b\|_r\|w\|_{r'}-\varepsilon
= \int_\Omega b(x)|u|\, dx-\varepsilon(\|b\|_r+1).
\end{multline}
Gathering \eqref{num} and \eqref{denom}, and having in mind that $v_n\in W_0^{1,1}(\Omega)$, we deduce
\begin{multline*}
\sigma(\Omega, a, b)\le 
\frac{\int_\Omega a(x)|\nabla v_n|\, dx}{\int_\Omega b(x)|v_n|\, dx}\\ 
\le \frac{\int_\Omega a(x)|Du|+\int_{\partial\Omega}a(x)|u|\, d\mathcal H^{N-1}+2\varepsilon}{\int_\Omega b(x)|u|\, dx-\varepsilon(\|b\|_r+1)}
\end{multline*}
for all $\varepsilon>0$. Therefore,
\[\sigma(\Omega, a, b)\le \frac{\int_\Omega a(x)|Du|+\int_{\partial\Omega}a(x)|u|\, d\mathcal H^{N-1}}{\int_\Omega b(x)|u|\, dx}.\]
Since $u\in BV(\Omega, a)$ is arbitrary, the result follows.
\end{proof}

We are now in a position to prove the main theorem of this paper when $a(x)$ is a Lipschitz weight.

\begin{cor}\label{prop:iden0}
The following statements hold for a Lipschitz weight $a(x)$:
\begin{enumerate}
\item There exists $\lim_{p\to1^+}\lambda_p(\Omega, a, b)=h(\Omega, a, b)$.
\item The weighted Cheeger constant coincides with the first eigenvalue of problem \eqref{limit_problem_intro2}.
\end{enumerate}
\end{cor}
\begin{proof}
The first statement is a consequence of Theorems \ref{prop:iden1} and \ref{ugua}. The second statement follows from Corollary \ref{comp-Chee}.
\end{proof}

\begin{rem}
It follows from Corollary \ref{prop:iden0} and \eqref{des-Chee} that the eigenfunction $u$ we have found in the previous sections minimizes the problem \eqref{first-chee}.
\end{rem}

\subsection{The main result}
The purpose of this subsection is to extend Corollary \ref{prop:iden0} (2) to the general case, even though we do not have Proposition \ref{minim} in this general setting.

\begin{thm}\label{prop:iden}
The weighted Cheeger constant is the first eigenvalue of problem \eqref{limit_problem_intro2}.
\end{thm}

\begin{proof}
First note that every $a\in W^{1,1}(\Omega)$ satisfying $a(x)\ge\mu$ has a representative which is lower semicontinuous. So  \cite[Theorem 2.2]{baldi} applies and  we may find a nondecreasing sequence of Lipschitz weights 
$\{a_k\}$ such that $\mu\le a_k(x)$ and $a(x)=\sup_ka_k(x)$ for all $x\in\Omega$. Then Beppo Levi's theorem on monotone convergence implies
\begin{equation}\label{conv-ak}
a_k\to a\quad\text{strongly in }L^1(\Omega).
\end{equation}
By Remark \ref{compar2}, we deduce that the sequence $\{h(\Omega, a_k, b)\}$ is nondecreasing. Set $\Upsilon=\sup_k h(\Omega, a_k, b)$. Appealing again to  Remark \ref{compar2}, we obtain $\Upsilon\le  h(\Omega, a, b)$. The rest of this proof is devoted to seeing the reverse inequality and that $\Upsilon$ is an eigenvalue of problem \eqref{limit_problem_intro2}.

On account of Theorem \ref{main0}, Proposition \ref{prop:minim} and Corollary \ref{prop:iden0}, there exists a nontrivial $u_k\in BV(\Omega, a_k, b)\cap L^\infty(\Omega)$ which minimizes the problem of Proposition \ref{prop:minim} and is a solution to problem
\begin{equation}\label{prob-k}
\left\{\begin{array}{cl}
\displaystyle -\dive\left(a_k(x)\frac{Du_k}{|Du_k|}\right)=h(\Omega, a_k, b)\, b(x)\frac{u_k}{|u_k|}& \text{in }\Omega\\[3mm]
u=0 & \text{on }\partial\Omega.
\end{array}\right.
\end{equation}
Therefore, there exist $z_k\in L^\infty(\Omega;\R^N)$ and $\gamma_k\in L^\infty(\Omega)$ satisfying
\begin{enumerate}
\item $\|z_k\|_\infty\le 1$ and $\|\gamma_k\|_\infty\le 1$
\item $(z_k, Du_k)=|Du_k|$ as measures and $\gamma_k u_k=|u_k|$ a.e. in $\Omega$.
\item $-\dive (a_k (x) z_k)=h(\Omega, a_k, b)\, b(x) \gamma_k$ in the sense of distributions
\item $[z_k, \nu]\in \sign (-u_k)$ $\mathcal H^{N-1}$-a.e. on $\partial\Omega$.
\end{enumerate}
Having in mind Remark \ref{test-sol} and the homogeneity of the minimization problem and \eqref{prob-k}, we may assume 
\[\int_\Omega b(x)|u_k|\, dx=1\,.\]
A further remark is in order, by Proposition \ref{acot}, we have $u_k\in L^\infty(\Omega)$ and
 \begin{equation}\label{acot-0}
 \|u_k\|_\infty\le \left(\frac{S\Upsilon\|b\|_r}{\mu}\right)^\frac{Nr}{r-N}\|u_k\|_1
 \end{equation}
 for all $k\in\N$.

Taking $u_k$ as a test function in \eqref{prob-k}, it follows from 
\[\int_\Omega a_k(x)|Du_k|+\int_{\partial\Omega}a_k(x)|u_k|\, d\mathcal H^{N-1}=
h(\Omega, a_k, b)\int_\Omega b(x)|u_k|\, dx\]
that 
\[\mu\left[\int_\Omega |Du_k|+\int_{\partial\Omega}|u_k|\, d\mathcal H^{N-1}\right]\le
\Upsilon\,.\]
Hence, the sequence $\{u_k\}$ is bounded in $BV(\Omega)$; there is a subsequence, still denoted as $\{u_k\}$, and there exists $u\in BV(\Omega)$ such that
\begin{equation}\label{conv-uk}
u_k\to u\quad\text{a.e. and strongly in }L^q(\Omega), \quad 1\le q<\frac{N}{N-1}\,.
\end{equation}
Moreover, \eqref{acot-0} and \eqref{conv-uk} imply $ \|u\|_\infty\le \left(\frac{S\Upsilon\|b\|_r}{\mu}\right)^\frac{Nr}{r-N}\|u\|_1$, and so $u$ belongs to $L^\infty(\Omega)$ and
\begin{equation}\label{acot-uk}
u_k{\buildrel * \over \rightharpoonup} u\quad\text{*-weakly in } L^\infty(\Omega).
\end{equation}

Our concern is now to check that $u$ is a solution to problem
\begin{equation}\label{prob-0}
\left\{\begin{array}{cl}
\displaystyle -\dive\left(a(x)\frac{Du}{|Du|}\right)=\Upsilon\, b(x)\frac{u}{|u|}& \text{in }\Omega\\[3mm]
u=0 & \text{on }\partial\Omega.
\end{array}\right.
\end{equation}

Fixed $\ell\in\N$, for every $k>\ell$ it holds
\begin{multline*}
\int_\Omega a_\ell(x)|Du_k|+\int_{\partial\Omega}a_\ell(x)|u_k|\, d\mathcal H^{N-1}\\
\le
\int_\Omega a_k(x)|Du_k|+\int_{\partial\Omega}a_k(x)|u_k|\, d\mathcal H^{N-1}\le\Upsilon.
\end{multline*}
Letting $k$ go to infinity, the lower semicontinuity with respect to the $L^1$-convergence implies
\[\int_\Omega a_\ell(x)|Du|+\int_{\partial\Omega}a_\ell(x)|u|\, d\mathcal H^{N-1}
\le\Upsilon.\]
Since it holds for every $\ell\in\N$, we have
\[\int_\Omega a(x)|Du|+\int_{\partial\Omega}a(x)|u|\, d\mathcal H^{N-1}
\le\Upsilon,\]
wherewith $u\in BV(\Omega, a)$. We also deduce that
\[\int_\Omega b(x)|u|\, dx=\lim_{k\to\infty}\int_\Omega b(x)|u_k|\, dx=1,\]
owing to \eqref{conv-uk} and \eqref{A3}. Choosing a further subsequence, if necessary, there exist 
\begin{enumerate}
\item $z\in L^\infty(\Omega; \R^N)$ such that $z_k{\buildrel * \over \rightharpoonup} z$ weakly-* in $L^\infty(\Omega; \R^N)$ and $\|z\|_\infty\le1$
\item $\gamma\in L^\infty(\Omega)$ such that $\gamma_k{\buildrel * \over \rightharpoonup} \gamma$ weakly-* in $L^\infty(\Omega)$ and $\|\gamma\|_\infty\le1$.
\end{enumerate}
Observe that, for every $\varphi\in C_c^\infty(\Omega)$, we get
\[\int_\Omega\varphi|u|\, dx=\lim_{k\to\infty}\int_\Omega\varphi |u_k|\, dx
=\lim_{k\to\infty}\int_\Omega\varphi \gamma_k u_k\, dx=\int_\Omega\varphi \gamma u\, dx,\]
so that $|u|=\gamma u$ a.e,. in $\Omega$. On the other hand, every $\varphi\in C_c^\infty(\Omega)$ satisfies
\[\int_\Omega a_k(x) z_k\cdot\nabla\varphi\, dx= h(\Omega, a_k, b)\int_\Omega b(x)\gamma_k\varphi\, dx,\]
for all $k\in\N$. Taking the limit as $k$ tends to $\infty$, it follows that
\[\int_\Omega a(x) z\cdot\nabla\varphi\, dx= \Upsilon\int_\Omega b(x)\gamma\varphi\, dx\,,\]
and so $-\dive(a(x)z)=\Upsilon b(x)\gamma$ holds in the sense of distributions.

Fix again $\ell\in\N$ and a nonnegative $\varphi\in C_c^\infty(\Omega)$. Taking $u_k\varphi$, with $k>\ell$, as test function in \eqref{prob-k}, we get
\begin{multline*}
\int_\Omega \varphi a_\ell(x)|Du_k|\le \int_\Omega \varphi a_k(x)|Du_k|\\
=-\int_\Omega  u_k a_k(x)z_k\cdot \nabla \varphi\, dx+
h(\Omega, a_k, b)\int_\Omega  \varphi b(x) |u_k| \, dx\,.
\end{multline*}
To pass to the limit as $k$ goes to infinity, we apply the lower semicontinuity of the left hand side, while in the last integral on the right hand side it is enough to keep in mind the convergences \eqref{conv-uk} and  that of $\{h(\Omega, a_k, b)\}$ to $\Upsilon$. Regarding the first term on the right hand side, observe that the pointwise convergence of $\{u_k\}$ and $\{a_k\}$ jointly with the convergences \eqref{conv-ak} and \eqref{acot-uk} lead to the convergence of $\{u_k a_k\}$ strongly in $L^1(\Omega)$. This fact and the *-weak convergence of $\{z_k\}$ allow us to let $k$ go to $\infty$ in this integral. We deduce
\[\int_\Omega \varphi a_\ell(x)|Du|\le 
-\int_\Omega  u\, a(x)z\cdot \nabla \varphi\, dx+
\Upsilon\int_\Omega  \varphi b(x) |u| \, dx\,.\]
Letting $\ell$ go to $\infty$ and recalling that $-\dive(a(x)z)=\Upsilon b(x)\gamma$ holds in the sense of distributions, we conclude that
\begin{multline*}
\int_\Omega \varphi a(x)|Du|\le 
-\int_\Omega  u\, a(x)z\cdot \nabla \varphi\, dx-
\int_\Omega  \varphi u\dive(a(x)z) \, dx\\
=\int_\Omega \varphi a(x) (z, Du)
\end{multline*}
is fulfilled for every nonnegative $\varphi\in C_c^\infty(\Omega)$. 
Therefore, $a(x)|Du|\le a(x) (z, Du)$ as measures, and assumption \eqref{A2} gives 
$|Du |\le (z,Du)$. Since the reverse inequelity is strightforward, we have seen $|Du |= (z,Du)$ as measures.

To check that $u$  is a solution to problem \eqref{prob-0}, 
it remains to prove that $-[z,\nu]\in \sign(u)$ on $\partial\Omega$. To this end, fix $\ell\in\N$, consider $k>\ell$ and take again $u_k$ as a test function, we infer
\begin{multline*}
\int_\Omega a_\ell(x)|Du_k|+\int_{\partial\Omega}a_\ell(x)|u_k|\, d\mathcal H^{N-1}\\
\le \int_\Omega a_k(x)|Du_k|+\int_{\partial\Omega}a_k(x)|u_k|\, d\mathcal H^{N-1}\\
=h(\Omega, a_k, b) \int_\Omega b(x)|u_k|\, dx\le \Upsilon \int_\Omega b(x)|u_k|\, dx\,.
\end{multline*}
The lower semicontinuity of the left hand side with respect to the $L^1$-convergence and \eqref{conv-uk} allows us to pass to the limit as $k$ goes to $\infty$:
\[\int_\Omega a_\ell(x)|Du|+\int_{\partial\Omega}a_\ell(x)|u|\, d\mathcal H^{N-1}
\le \Upsilon \int_\Omega b(x)|u|\, dx\,.\]
Letting now $\ell$ go to $\infty$, we obtain
\begin{multline*}
\int_\Omega a(x)|Du|+\int_{\partial\Omega}a(x)|u|\, d\mathcal H^{N-1}
\le \Upsilon \int_\Omega b(x) \gamma u\, dx\\
= -\int_\Omega u \dive(a(x)z)=\int_\Omega a(x) (z,Du)-\int_{\partial\Omega}a(x)u[z,\nu]\, d\mathcal H^{N-1}
\end{multline*}
due to the Green formula. Simplifying we arrive at
\[\mu\int_{\partial\Omega} |u|+u[z,\nu]\, d\mathcal H^{N-1}\le\int_{\partial\Omega}a(x)\big(|u|+u[z,\nu]\big)\, d\mathcal H^{N-1}
\le 0\]
and so $|u|+u[z,\nu]=0$ $H^{N-1}$-a.e.  on $\partial\Omega$.

We have proven that $u$ is a solution to problem \eqref{prob-0} wherewith $\Upsilon$ is an eigenvalue.
Now, Corollary \ref{comp-Chee} implies $h(\Omega, a, b)\le\Upsilon$ and we are done.
\end{proof}

  \bibliographystyle{abbrv}
\bibliography{name}
\end{document}